\documentclass[12pt,reqno]{amsart}

\usepackage{amssymb,amsmath,amsthm,amstext,amscd}
\usepackage{amsfonts}
\usepackage{mathrsfs}
\usepackage[english]{babel}

\usepackage[pdftex]{color}

\newtheorem{theorem}{Theorem}[section]
\newtheorem{lemma}[theorem]{Lemma}

\theoremstyle{definition}
\newtheorem{assumption}{Assumption}[section]
\newtheorem{definition}{Definition}[section]

\theoremstyle{remark}
\newtheorem{remark}{Remark}[section]

\newcommand\bR{\mathbb{R}}

\newcommand\D{\partial}

\begin{document}

\title[]{A Note On Degenerate Stochastic Integro-Differential Equations
}
\author{Konstantinos Dareiotis}
\begin{abstract}
In the present article, solvability in Sobolev spaces is investigated for a class of degenerate stochastic integro-differential equations of parabolic type. Existence and uniqueness is obtained, and estimates are given for the solution.
\end{abstract}
\maketitle

\section{Introduction}

 In the present paper, solvability in $L_2$ spaces of  stochastic partial integro-differential equations is investigated, under the lack of super-parabolicity. Equations of this type arise in non-linear filtering of jump-diffusion processes. 
 
 In the case of stochastic partial differential equations (SPDEs), the first results of solvability of the corresponding equations under the lack of super-parabolicity appeared in \cite{KR}. There the solution is also shown to belong in $L_p$ spaces. More recently, a gap from \cite{KR} is filled and the results are extended to systems of SPDEs in \cite{GGK}. While the present article was prepared for submission, \cite{J} appeared with similar results. However we decided to upload the article, since the results are obtained independently, the proof of the  main estimates is different, and we obtain $E\sup_{t\leq T}\|\cdot\|^2_m-$estimates for the solution, which allow us to conclude that the solution is weakly c\'adl\'ag in $H^m$ (see below for definitions).
 
 In the present paper we focus on estimating the terms appearing due to the integral operators and the jumps, and once the proper estimates are obtained, then existence, uniqueness, as well as estimates for the solution follow by the technique introduced in \cite{KR}. 
 
Finally let us conclude by introducing some notation that will be used through the paper, and finally state the exact problem that we are interested in. Let $T$ be a positive real number. We consider a filtered probability space $(\Omega, \mathscr{F}, (\mathscr{F}_t)_{t \leq T}, P)$, with the filtration satisfying the usual conditions. On this probability space we consider a sequence $(w^k_t)_{k=1}^\infty$ of independent $\mathscr{F}_t-$Wiener processes.  We also consider a measure space $(Z, \mathscr{Z}, \mu)$, where  $\mu$ is a L\'evy measure,  and a Poisson random measure $N(dz,dt)$ on $Z\times [0,T]$, defined on our probability space, such that $N(dz,dt)- \mu(dz)dt$ is a martingale measure. We will denote the progressive $\sigma-$field on $\Omega \times [0,T]$ by $\mathscr{P}$, and if $X$ is a topological space, $\mathscr{B}(X)$ will denote the Borel $\sigma-$field on $X$.  Let $d$ be a positive integer. For an integer $m\geq 0$, $H^m$ will denote the Sobolev space of functions in $L_2(\bR^d)$, having distributional derivatives of order $m$ in $L_2(\bR^d)$ and we will denote by $(\cdot, \cdot)_m$ and $\|\cdot\|_m$ the inner product and the norm respectively in $H^m$. If $X$ is a hilbert space, $H^m(X)$ will denote the corresponding space of $X-$valued functions on $\bR^d$.  We introduce also the following notation for function spaces
$\mathbb{H}^m:= L_2(\Omega, \mathscr{F}_0, H^m)$,  $\mathfrak{H}^m:= L_2(\Omega \times [0,T], \mathscr{P}; H^m)$, and if $X$ is a Hilbert space, then $\mathfrak{H}^m(X):= L_2(\Omega \times [0,T], \mathscr{P}; H^m(X))$.  For $i,j \in \{1,...,d\}$ we will use the notation $\D_i :=\D /  \D x_i$, $\D_{ij}:= \D_i \D_j$, and for a multi-index $\alpha = (\alpha_1,..., \alpha_d)$,  $\D^\alpha:= (\D^{|\alpha_1|}/ \D x_1)...  (\D^{|\alpha_d|}/ \D x_d)$. Also for foa a matrix $A$, we will denote the determinant by $detA$, and for a map $b: \bR^d \to \bR^d$, $Db$ will stand for the Jacobian. For  $i,j \in \{1,...,d\}$ and  $k\in \mathbb{N}$, we are given functions $a^{ij}, a^{i}, a, \sigma^{ik}, \sigma^k$  defined on $\Omega\times[0,T]\times \bR^d$ with values in $\bR$, that  are $\mathscr{P} \otimes \mathscr{B}(\bR^d)-$measurable. Also, we are given a  real-valued function $c=(c^1,...,c^d)$, defined on $\Omega\times [0,T]\times \bR^d \times Z$, with values in $\bR^d$, which is $\mathscr{P} \otimes \mathscr{B}(\bR^d)\otimes \mathscr{Z}-$measurable.  

On $[0,T] \times \bR^d$ we consider the following equation
$$
du_t=(L_tu_t+I_tu_t+f_t)dt+\sum_{k=1}^ \infty (M^k_tu_t+g^k_t)dw^k_t
$$
\begin{equation}                                    \label{eq:main equation}
+\int_Z (H_t(z)u_t +h_t(z)) \tilde{N}(dz,dt),
\end{equation}
with initial condition 
$
u_0= \psi,
$
where the operators are given by 
$$
 L_t v (x)= \sum_{i,j=1}^d a^{ij}_t(x)\D_{ij} v(x)+\sum_{i=1}^da^i_t(x)v(x)+a_t(x)v(x),
 $$
 $$
  \ M^k_tv(x)=\sum_{i=1}^d \sigma^{ik}_t (x)\D_iv(x)+\sigma^k_t(x) v(x),
 $$
 $$
  \ H_t(z)v(x)=v(x+c_t(z,x))-v(x),
 $$
 and
 $$
 I_tv(x)= \int_Z v(x+c_t(z,x))-v(x)- c_t(z,x)\cdot \nabla v(x) \mu(dz).
 $$  
  The structure of the article is as follows. In Section 2, we state the assuptions and the main result, in Section 3 we give the main estimates, and in Section 4 we use these estimates to prove our main theorem.
  \section{Main Results}
  In this section we state our assumption and the main results. Let $m \geq 1$ be an integer, and set $\mathfrak{m}:=\max(m,2)$. Let also $K>0$ be a constant.
 \begin{assumption}                     \label{as: regularity of a}
The functions $a^{ij}$ are $\mathfrak{m}$ times continuously  differentiable in $x$ and their derivatives up to order $\mathfrak{m}$ are bounded by $K$. The functions $a^i,a$ are $m$ times continuously  differentiable and their derivatives up to order $m$ are bounded by $K$. The functions $\sigma^{i}=(\sigma^{ik})_{k=1}^\infty$ and $\sigma=(\sigma^k)_{k=1}^\infty$ are $l_2-$valued, $m+1$ times continuously differentiable, and their derivatives up to order $m+1$ are bounded by $K$. 

 \end{assumption}
\begin{assumption}        \label{as:regularity of c}
The function $c$ is   $m+1$ times continuously differentiable in $x$,  and there exists a function $\bar{c}  \in L_2(Z, \mu(dz))$, such that 
for any multi-index $\alpha$, with $|\alpha| \leq m+1$, we have for all $\omega, t, x,$ and $z$,
$$
|\D^\alpha c_t(z,x)| \leq \bar{c}(z)\wedge K.
$$
Moreover, for all $(\omega,t,x,z,\theta) \in \Omega \times [0,T] \times\bR^d \times Z \times [0,1]$
$$
K^{-1} \leq |det( \mathbb{I}+\theta Dc_t(x,z))|.
$$
\end{assumption}
From now on, for $\theta\in [0,1]$,  will be using the notation $T_{t,\theta,z}(x):=x+\theta c_t(x,z)$. Notice that under Assumption \ref{as:regularity of c}, for fixed $\omega,t,\theta, z$, the map $T$ is a $C^{m+1}-$diffeomorphism on $\bR^d$, and we will denote its inverse by $J_{t,\theta,z}(x)$.
\begin{assumption}
$f \in \mathfrak{H}^{m}$, $g \in \mathfrak{H}^{m+1}(l_2)$, $h \in \mathfrak{H}^{m+1}(L_2(Z))$, and $\psi \in \mathbb{H}^m$. 
\end{assumption}
\begin{assumption}                         \label{as: parabolicity}
For all $\omega,t,x$ and all $\xi=(\xi_1, ..., \xi_d) \in \bR^d$
$$
\sum_{i,j=1}^d a^{ij}_t(x)\xi_i \xi_j -\frac{1}{2} \sum_{i,j=1}^d\sum_{k=1}^\infty \sigma^{ik}_t(x) \sigma^{jk}_t(x) \xi_i\xi_j \geq 0.
$$
\end{assumption}
\begin{definition}
 By a solution of equation \eqref{eq:main equation}, we mean a function $u \in \mathfrak{H}^1$ such that for each $\phi \in C_c^\infty(\bR^d)$, for $dP \times dt$ almost all $(\omega,t)\in \Omega \times [0,T]$ one has the equality
 $$
 (u_t,\phi)_0
 $$
 $$
 =(\psi,\phi)_0+\int_0^t -(\D_j\phi, a^{ij}_s \D_iu_s)_0+(\phi, a^i_s\D_iu_s-\D_ja^{ij}\D_iu_s+a_s u_s+f_s)_0 ds
 $$
 $$
 +\int_0^t \int_0^1 (\theta-1) 
\int_Z \int_{\bR^d} 
\D_iu_s(x+\theta c_s(x,z)) \D_j
( q^{ij}_s(x,z,\theta) \phi(x))\,dx \mu(dz)d\theta ds
 $$
 $$
 +\int_0^t (\phi, \sigma^{ik}\D_i u_s + \sigma^ku_s+g^k_s u_s)_0 dw^k_s+\int_0^t \int_Z (H_s(z)u_s+h_s(z), \phi)_0 \tilde{N}(dz,ds),
 $$
 where the summation is understood with respect to integer valued repeated indeces, and 
$q^{ij}_t(x,z, \theta):= \sum_{l=1}^d c^l_t(x,z)c^i_t( x,z)  \D_lJ^j_{\theta, t,z}(T_{\theta,t,z}(x))$.
\end{definition}
We will denote the expression of the right hand side of the definition of the solution by $R(t,a^{ij}, u,\phi)$, in order to ease the notation in the proof of our main theorem in Section 4. Also notice that we do not require the integrants in the jump term to be predictable since the compensator of the martingale is continuous.
\begin{remark}
As in \cite{KR}, if $u\in \mathfrak{H}^n$, for $n \geq1$, is a a solution of equation \eqref{eq:main equation}, then there exists a function $\bar{u}$, wich is  equal to $u$ for almost every $(\omega,t)$, is strongly c\'adl\'ag in $t$ as process with values in $H^{n-1}$, and the equality in the definition of solution holds almost surely, for all $t \in [0,T]$. This follows from the main theorem in \cite{G2}.
\end{remark}

From now on,  for an $l_2$ valued function $g$, with abuse of notation we will be  writing $\|g\|_m^2:= \sum_{k=1}^ \infty \|g^k\|^2_m$.
\begin{theorem}                                \label{thm: main theorem}
 Let $m \geq 1$ be an integer. Under Assumptions \ref{as: regularity of a} -\ref{as: parabolicity}, there exists a unique solution $u$ of equation \eqref{eq:main equation}. Moreover, $u$ is a strongly c\'adl\'ag processes with values in $H^{m-1}$, weakly c\'adl\'ag in  $H^{m}$, and the following estimate holds,
$$
 E\sup_{t\leq T} \|u_t\|^2_m 
 $$
 \begin{equation}                                \label{eq: estimate for the solution}
  \leq NE\|\psi\|^2_m+NE\int_0^T\left( \|f_t\|^2_m+\|g_t\|^2_{m+1}+\int_Z\|h_t(z)\|_{m+1}^2 \mu(dz) \right) dt,
\end{equation}
 where $N$ is a constant depending only on $m,d,T$ and $K$.
 \end{theorem}
\section{Auxiliary Results}
\begin{lemma}                                \label{lem: L1 estimate}
Suppose Assumption \ref{as:regularity of c} holds. There exists a constant $N:=N(m,d, \mu)$ such that for any $t, \omega$ and  $v\in W^2_1(\bR^d)$ we have
$$
\int_{\bR^d} \int_Z [v(x+c_t(x,z))-v(x)-c_t(x,z)\cdot \nabla v(x) ]\mu(dz) dx \leq N \|v\|_{L_1}.
$$
\end{lemma}
\begin{proof}
By Taylor's theorem we have
$$
\bar{v}(x,z):=v(x+c_t(x,z))-v(x)-c_t(x,z)\cdot \nabla v(x)
$$
$$
=\sum_{i,j=1}^d\int_0^1 (1-\theta) \D_{ij}v(T_{\theta,t,z}(x))c^{i}_t(x,z)c^j_t(x,z) d\theta.
$$
Integrating over $\bR^d$ and changing variables gives
$$
\int_{\bR^d} \bar{v}(x,z) dx =\sum_{i,j=1}^d\int_0^1(1-\theta) \int_{\bR^d}
 \D_{ij}v(x)q^{ij}_t(\theta,x,z) dx d\theta,
$$
where $q^{ij}_t(\theta,x,z):=c^{i}_t(J_{\theta,t,z}(x),z)c^j_t(J_{\theta,t,z}(x),z) |det D J_{\theta,t,z}(x)|$.
By integration by parts and using Assumption \ref{as:regularity of c}, we obtain

$$
\int_{\bR^d} \bar{v}(x,z) dx \leq N |\bar{c}(z)|^2\|v\|_{L_1},
$$
Hence, by integrating over $Z$, we obtain the desired inequality.
\end{proof}

In the calculations later on, we  drop some arguments from the functions to ease the notation. We will be writing for example $u(x+c)$ instead of $u_t(x+c_t(x,z))$.

 For a real function  $u$ on $\Omega\times [0,T] \times \bR^d$, and functions $F: \mathbb{Z}^d \to \bR$, $m: \mathbb{Z}^d \to \mathbb{N}$ we will write $\mathscr{D}_n(u,F,m)$ for the collection of all functions $v$, such that 
 $$
 v= 
  \sum_{|\zeta| \leq n}  F(\zeta) D^\zeta u(x+c)
 $$
 \begin{equation}                                       \label{eq:comp form}
 \times  \sum_{k=1}^{m(\zeta)} \prod_{i=1}^n \left(\D^{\alpha^k_i(\zeta)} (x^{f^k_i(\zeta)}+c^{f^k_i(\zeta)}) \right)^{\delta^k_i(\zeta)}  \prod_{i=1}^n\left( \D^{\beta^k _i(\zeta)} c^{g^k_i(\zeta)} \right)^{{\delta'}^k_i(\zeta)},
 \end{equation}
 where $ \delta^k_i(\zeta), {{\delta'}^k_i(\zeta)}\in \{0,1\}$,   
 $f^k_i(\zeta), g^k_i(\zeta) \in \{0,...,d\}$,  $\alpha^k_i(\zeta)$, $\beta^k_i(\zeta)\in \mathbb{N}^d$,  and for each $\zeta$ and $k$, there exist $i,j \in \{0,n\}$, $i\neq j$ such that $ {{\delta'}^k_i(\zeta)}= {{\delta'}^k_j(\zeta)}=1$.
\begin{remark}                                      \label{rem: |v|, |v|2}
 If $v \in \mathscr{D}_n(u, F, m)$, there exists a constant $N$ depending only on $K,d,n,F,$ and $m$, such that if $z \in \{ \bar{c}(z) \leq 1\}$, then for any $\omega,t,x,$
 $$
  |v|^2 \leq N |\bar{c}(z)|\sum_{|\zeta |\leq n} |\D^\zeta u(x+c)|^2
 $$
 \begin{equation}                             
 |v| \leq N |\bar{c}(z)|^2\sum_{|\zeta |\leq n} |\D^\zeta u(x+c)|, \ \  |v|^2 \leq N |\bar{c}(z)|^2\sum_{|\zeta |\leq n} |\D^\zeta u(x+c)|^2
 \end{equation}
 \end{remark}
 \begin{lemma}                                 \label{lem: form composition}
 For any multi-intex $\alpha$ with  $n=|\alpha|\geq 1$ there exists functions $F,m$ depending only on $\alpha$ such that for any $u \in H^n$, 
 $$
 \D^\alpha (u(x+c))=\D^\alpha u(x+c) + \sum_{\rho=1}^d \sum_{\begin{subarray}{l}\beta \neq 0 \\ \beta \leq \alpha  \end{subarray}}  {\alpha \choose \beta } \D^\beta c^\rho \D^{\alpha- \beta} \D_\rho u(x+c)+u^{(\alpha)},
 $$
 where $u^{(\alpha)} \in \mathscr{D}_n(u,F,m)$.
 
 \end{lemma}
 \begin{proof}
 It is easy to see that this holds for any multi-index $\alpha$ with $1\leq |\alpha| \leq 2$. We proceed by induction and  we assume that the equality holds for some $\alpha$ with $|\alpha| \geq 2$ . For any multi-index $\gamma$, we  set $\tilde{\gamma}=\gamma+ e_i$ for some fixed $i \in \{1,...,d\}$, where $e_i$ is the unit vector in the $ i-$th direction. It is easy to see that  $\D_i u^{(\alpha)} \in  \mathscr{D}_{n+1} (u,F',m')$, where the functions $F', m'$ depend only on $\tilde{\alpha}$.
 Hence it suffices to show that there exist functions $F'',m''$ depending only on $\tilde{\alpha}$, and $v \in \mathscr{D}_{n+1}(u,F'',m'')$ such that
 $$
A:=\D_i  \left(\D^\alpha u(x+c) + \sum_{\rho=1}^d \sum_{\begin{subarray}{l}\beta \neq 0 \\ \beta \leq \alpha  \end{subarray}}  {\alpha \choose \beta } \D^\beta c^\rho \D^{\alpha- \beta} \D_\rho u(x+c)\right)
 $$ 
 \begin{equation}                      \label{eq: desired eq}
= \D^{\tilde{\alpha}} u(x+c) + \sum_{\rho=1}^d \sum_{\begin{subarray}{l}\beta \neq 0 \\ \beta \leq \tilde{\alpha} \end{subarray}}  {\tilde{\alpha} \choose \beta } \D^\beta c^\rho \D^{\tilde{\alpha}- \beta} \D_\rho u(x+c)+v.
 \end{equation}
 We have 
 \begin{equation}                           \label{eq: B first}
 B:=\D_i  \left(\D^\alpha u(x+c)\right)=\D^{\tilde{\alpha}} u(x+c)+\sum_{\rho=1}^d \D_\rho \D^{\alpha} u(x+c) \D_i c^\rho.
 \end{equation}
 We also have 
 $$
C:=\D_i\left( \sum_{\rho=1}^d \sum_{\begin{subarray}{l}\beta \neq 0 \\ \beta \leq \alpha \end{subarray}} { \alpha \choose \beta } \D^\beta c^\rho \D^{\alpha- \beta} \D_\rho u(x+c)\right)
 $$
 $$
 =\sum_{\rho=1}^d \sum_{\begin{subarray}{l}\beta \neq 0 \\ \beta \leq \alpha \end{subarray}} { \alpha \choose \beta } \D^{\tilde{\beta}} c^\rho \D^{\alpha- \beta} \D_\rho u(x+c)
$$
$$
 +
 \sum_{\rho=1}^d \sum_{\begin{subarray}{l}\beta \neq 0 \\ \beta \leq \alpha \end{subarray}} { \alpha \choose \beta } \D^\beta c^\rho \sum_{k=1}^d \D_k \D^{\alpha- \beta} \D_\rho u(x+c) \D_i(x^k+c^k)
 $$
 $$
 =\sum_{\rho=1}^d \sum_{\begin{subarray}{l}\beta \neq 0 \\ \beta \leq \alpha \end{subarray}} { \alpha \choose \beta } \D^{\tilde{\beta}} c^\rho \D^{\alpha- \beta} \D_\rho u(x+c)
 $$
 $$
 +\sum_{\rho=1}^d \sum_{\begin{subarray}{l}\beta \neq 0 \\ \beta \leq \alpha \end{subarray}} { \alpha \choose \beta } \D^\beta c^\rho \D_\rho \D^{\tilde{\alpha}- \beta}  u(x+c)+v, 
$$
where

$$
v= \sum_{\rho=1}^d \sum_{\begin{subarray}{l}\beta \neq 0 \\ \beta \leq \alpha \end{subarray}} { \alpha \choose \beta } \D^\beta c^\rho \sum_{k=1}^d \D_k \D^{\alpha- \beta} \D_\rho u(x+c) \D_i c^k,
$$
and  is in the class $\mathscr{D}_{n+1}(u,\bar{F}, \bar{m})$, for some functions   $\bar{F}, \bar{m}$ depending only on $\tilde{\alpha}$.
Also 
$$
\sum_{\rho=1}^d \sum_{\begin{subarray}{l}\beta \neq 0 \\ \beta \leq \alpha \end{subarray}} { \alpha \choose \beta } \D^{\tilde{\beta}} c^\rho \D^{\alpha- \beta} \D_\rho u(x+c)
$$
 $$
 =\sum_{\rho=1}^d \sum_{\begin{subarray}{l} 2 \leq |\beta| \\ e_i \leq \beta \leq \tilde{\alpha} \end{subarray}} { \alpha \choose \beta- e_i } \D^\beta c^\rho \D^{\tilde{\alpha}- \beta} \D_\rho u(x+c)
 $$
 Hence 
 $$
C= \sum_{\rho=1}^d \sum_{\begin{subarray}{l} 2 \leq |\beta| \\ e_i \leq \beta \leq \alpha \end{subarray}}\Big[ { \alpha \choose \beta- e_i }+ { \alpha \choose \beta} \Big]\D^\beta c^\rho \D^{\tilde{\alpha}- \beta} \D_\rho u(x+c)
 $$
 $$
 +\sum_{\rho=1}^d \D^{\tilde{\alpha}}c^\rho \D_\rho u(x+c)+\sum_{\rho=1}^d \sum_{\begin{subarray}{l}\beta \neq 0 \\ \beta \leq \alpha \\ \beta_i=0 \end{subarray}} { \alpha \choose \beta } \D^\beta c^\rho \D_\rho \D^{\tilde{\alpha}- \beta}  u(x+c) 
 $$
 $$
 +\sum_{\rho=1}^d { \alpha \choose e_i } \D_i c^\rho \D^\alpha \D_\rho u(x+c) +v
 $$
 Notice that 
 $$
  { \alpha \choose \beta- e_i }+ { \alpha \choose \beta}=   { \tilde{\alpha} \choose \beta},
  $$
  and also, if $\beta_i=0$, then 
  $$
  { \alpha \choose \beta}={\tilde{\alpha} \choose \beta}.
  $$ 
  Therefore 
  $$
  C=\sum_{\rho=1}^d \sum_{\begin{subarray}{l} 2 \leq |\beta| \\ e_i \leq \beta \leq \tilde{\alpha} \end{subarray}} { \tilde{\alpha} \choose \beta} \D^\beta c^\rho \D^{\tilde{\alpha}- \beta} \D_\rho u(x+c)
  $$
  $$
  +\sum_{\rho=1}^d \sum_{\begin{subarray}{l}\beta \neq 0 \\ \beta \leq \tilde{\alpha} \\ \beta_i=0 \end{subarray}} { \tilde{\alpha} \choose \beta } \D^\beta c^\rho \D_\rho \D^{\tilde{\alpha}- \beta}  u(x+c) 
  $$
  \begin{equation}                              \label{eq: B}
 +\sum_{\rho=1}^d { \alpha \choose e_i } \D_i c^\rho \D^\alpha \D_\rho u(x+c) +v.
\end{equation}
Consequently, by summing \eqref{eq: B first} and \eqref{eq: B} we obtain \eqref{eq: desired eq}. This finishes the proof.
 \end{proof}
 For a multi-index $\alpha$, with $|\alpha|=n$,  and a function $u \in H^{n+2}$, let us define the quantity
  
$$
 \mathscr{G}_t^\alpha (u):= 
  \int_Z \| \D^\alpha \left( u(x+c_t(x,z))-u(x) \right) \|^2_0 \mu(dz).
$$
$$
+ 2\left(\int_Z \D^\alpha\left( u(x+c_t(x,z))-u(x)-c_t(x,z)\cdot \nabla u(x)\right) \mu(dz),  \D^\alpha u(x)  \right)_0
  $$
 \begin{lemma}
 Let Assumption \ref{as:regularity of c} hold. Then there exists a constant $N$ depending only on $K,d$ and $m$, such that for any $u\in H^{m+2}$, and for any multi-index $\alpha$, with $|\alpha|\leq m$, we have
 \begin{equation}                          \label{eq: main estimate}
 \mathscr{G}_\alpha(u) \leq N \|u\|^2_m.
 \end{equation}
 \end{lemma}
 \begin{proof}
 A simple calculation shows that 
\begin{equation}                                \label{eq: integral G alpha}
 \mathscr{G}_\alpha(u)=\int_{\bR^d} \int_Z [\D^\alpha \left( u(x+c)\right)]^2-[\D^\alpha u(x)]^2-2\D^\alpha[c\nabla u(x)] \D^\alpha u(x) \mu(dz) dx.
 \end{equation}

 By Lemma \ref{lem: form composition} we have 
$$
[\D^\alpha \left( u(x+c)\right)]^2=[\D^\alpha u(x+c)]^2
$$
$$
=2\D^\alpha u(x+c)\sum_{\rho=1}^d \sum_{\begin{subarray}{l}\beta \neq 0 \\ \beta \leq \alpha  \end{subarray}}  {\alpha \choose \beta } \D^\beta c^\rho \D^{\alpha- \beta} \D_\rho u(x+c) +2\D^\alpha u(x+c) v 
$$
 \begin{equation}                             \label{eq: comp squared}
 \left(\sum_{\rho=1}^d \sum_{\begin{subarray}{l}\beta \neq 0 \\ \beta \leq \alpha  \end{subarray}}  {\alpha \choose \beta } \D^\beta c^\rho \D^{\alpha- \beta} \D_\rho u(x+c) + v \right)^2 ,
 \end{equation}
 where $v\in \mathscr{D}_n(u,F,m)$ for some functions $F,m$ depending only on $\alpha$.
We also have
$$
2\D^\alpha[c\nabla u(x)] \D^\alpha u(x)= c \cdot \nabla (\D^\alpha u(x))^2
$$
\begin{equation}
+ 2\D^\alpha u(x)\sum_{\rho=1}^d \sum_{\begin{subarray}{l}\beta \neq 0 \\ \beta \leq \alpha  \end{subarray}}  {\alpha \choose \beta } \D^\beta c^\rho \D^{\alpha- \beta} \D_\rho u(x)
\end{equation}
Hence the integrant in \eqref{eq: integral G alpha} is equal to 
\begin{equation}                           \label{eq: integrant}
[\D^\alpha u(x+c)]^2-[\D^\alpha u(x)]^2- c \cdot \nabla (\D^\alpha u(x))^2+A_1+A_2 
\end{equation}
where
$$
A_1=2\sum_{\rho=1}^d \sum_{\begin{subarray}{l}\beta \neq 0 \\ \beta \leq \alpha  \end{subarray}}  {\alpha \choose \beta } \D^\beta c^\rho[ \D^{\alpha- \beta}\D_\rho u(x+c) \D^\alpha u(x+c)-\D_\rho \D^{\alpha- \beta} u(x) \D^\alpha u(x)]
$$
and 
$$
A_2=\left(\sum_{\rho=1}^d \sum_{\begin{subarray}{l}\beta \neq 0 \\ \beta \leq \alpha  \end{subarray}}  {\alpha \choose \beta } \D^\beta c^\rho \D^{\alpha- \beta} \D_\rho u(x+c) + v \right)^2 +2\D^\alpha u(x+c) v .
$$
By Taylor's formula we have
$$
\int_{\bR^d} \D^\beta c^\rho [\D_\rho u(x+c) \D^\alpha u(x+c)-\D_\rho u(x) \D^\alpha u(x)]dx
$$
$$
=\int_0^1\sum_{i=1}^d \int_{\bR^d}\D_i (\D_\rho\D^{\alpha- \beta} u \D^\alpha u )(T_{t,z,\theta}(x)) c^i_t(x,z) \D^\beta c^\rho_t(x,z)  dxd\theta
$$
$$
=\int_0^1\sum_{i=1}^d \int_{\bR^d}\D_i (\D_\rho\D^{\alpha- \beta} u \D^\alpha u )(x) c^i_t(J_{t,z,\theta}(x),z) \D^\beta c^\rho_t(J_{t,z,\theta}(x),z)
$$
$$
\times |detD J_{\theta,t,z}(x)|  dxd\theta
\leq N |\bar{c}(z)|^2 \|u\|_m^2
$$
where for  the last  inequality we have used integration by parts, Assumption \ref{as:regularity of c}, and H\"older's inequality. Therefore
\begin{equation}                               \label{eq: A1}
\int_{Z}\int_{\bR^d} A_1 dx dz \leq N  \|u\|_m^2.
\end{equation}
It is also easy to see that 
$$
|A_2| \leq N\left(|\bar{c}(z)|^2 \sum_{\rho=1}^d \sum_{\begin{subarray}{l}\beta \neq 0 \\ \beta \leq \alpha  \end{subarray}} |\D^{\alpha- \beta} \D_\rho u(x+c)|^2+|v|^2+|\D^\alpha u(x+c)||v|\right),
$$
and by Remark \ref{rem: |v|, |v|2} we obtain
$$
|A_2| \leq N|\bar{c}(z)|^2 \left(\sum_{|\zeta |\leq n} |\D^\zeta u(x+c)|^2+|\D^\alpha u(x+c)|\sum_{|\zeta |\leq n} |\D^\zeta u(x+c)|\right).
$$
Consequently,
\begin{equation}                                    \label{eq: A2}
\int_{ Z} \int_{\bR^d} |A_2| dx \mu(dz)\leq N  \|u\|_m^2.
\end{equation}
Therefore by integrating \eqref{eq: integrant} over $\bR^d$ and $Z$, Lemma \ref{lem: L1 estimate} combined with \eqref{eq: A1} and \eqref{eq: A2}, leads to \eqref{eq: main estimate}.
 \end{proof}

 \begin{assumption}                            \label{as: non-de regularity}
$$ $$
\begin{itemize}
\item[(i)] The functions $a^{ij}, \sigma^{ik}, c$ are $m$ times continuously differentiable in $x$,  and their derivatives up to order $m$ are bounded in magnitude by $K$. Moreover for any multi-index $\alpha$, with $|\alpha| \leq m$, 
 $$
 |\D^\alpha c_t(x,z)| \leq |\bar{c}(z)| \wedge K
 $$
 \item[(ii)] For any $(\omega,t,x,z,\theta) \in \Omega \times [0,T] \times\bR^d \times Z \times [0,1],$
 $$
 K^{-1} \leq |det(\mathbb{I}+ \theta D c_t(x,z))|
 $$
 \item[(iii)] $f \in \mathfrak{H}^{m-1}, \ g \in  \mathfrak{H}^m(l_2)$, $h \in \mathfrak{H}^m(L_2(Z))$ and $\psi \in \mathbb{H}^m$.
 \end{itemize} 
 \end{assumption}
 \begin{assumption}                          \label{as: non-de parabolicity}
 There exists a constant $\lambda>0$, such that for any $\xi \in \mathbb{R}^d$, and for all $\omega,t$ and $x$, we have
 $$
 \sum_{i,j=1}^da^{ij}_t(x)\xi_i \xi_j  -\sum_{i,j=1}^d \sum_{k=1}^\infty \frac{1}{2}  \sigma^{ik}_t(x)\sigma ^{jk}_t(x)\xi_i \xi _j 
\geq \ \lambda |\xi |^2.
 $$
 \end{assumption}
 The following theorem is a consequence of Theorems 2.9-2.10 from \cite{G3}.
 \begin{theorem}
 Under Assumptions \ref{as: non-de regularity}-\ref{as: non-de parabolicity}, there exists a unique solution $u$ of equation \eqref{eq:main equation}. Moreover, $u$ is a c\'adl\'ag processes with values in $H^m$, it belongs to $H^{m+1}$ for $dP \times dt$ almost every $(\omega,t)$, and the following estimate holds
$$
 E\sup_{t\leq T} \|u_t\|^2_m+E\int_0^T \|u_t\|^2_{m+1} dt
 $$
  \begin{equation}                          \label{eq: main estimate non-de}
  \leq NE\|\psi\|^2_m+NE\int_0^T\left( \|f_t\|^2_{m-1}+\|g_t\|^2_m+\int_Z\|h_t(z)\|_m^2 \mu(dz) \right) dt,
 \end{equation}
 where $N$ is a constant depending only on $m,d,T,K$ and $\lambda$.
 \end{theorem}
 \section{Proof of Theorem \ref{thm: main theorem}}
 \begin{proof}
 First we assume that the functions $a^{ij}$ are smooth, and we replace them by $a^{ij}(\varepsilon):= a^{ij}+\varepsilon\delta_{ij}$, for some $\varepsilon>0$. Then the modified equation has a unique solution $u^\varepsilon\in \mathfrak{H}^{m+2}$, which is c\'adl\'ag in $H^{m+1}$. Then for a multi-index $\alpha$, with $|\alpha| \leq m$ we can differentiate the equation, use It\^o's formula for the square of the $L_2(\bR^d)-$norm (see \cite{G3}) and sum over all $|\alpha| \leq m$, to obtain
 $$
\|u^\varepsilon_t\|^2_m= \|\psi\|^2_m+\int_0^t I(m,u^\varepsilon_s,f_s,g_s)+\sum_{|\alpha| \leq m} \mathscr{G}^\alpha_s(u^\varepsilon)  ds 
 $$
 $$
 +2\sum_{|\alpha| \leq m} \int_0^t \int_Z\left(\D^\alpha H_s(z)u^\varepsilon_s, \D^\alpha h_s(z) \right)_0 +\|\D^\alpha h_s(z)\|^2_0 \mu(dz) ds
 $$
 $$
 +2\sum_{k=1}^ \infty \int_0^t (u^\varepsilon_s, M^ku^\varepsilon_s+ g^k_s)_m \ dw^k_s
 $$
 $$
 +\sum_{|\alpha| \leq m} \int_0^t \int_Z\|\D^\alpha(u^\varepsilon(x+c))\|_0^2-\|\D^\alpha u^\varepsilon (x)\|_0^2  \tilde{N}(dz,ds).
 $$
 $$
 +\sum_{|\alpha| \leq m} \int_0^t \int_Z 2 (\D^\alpha(u^\varepsilon_s(x+c)), \D^\alpha h_s(z))_0+\|\D^\alpha h_s(z)\|^2_0 \tilde{N}(dz,ds),
 $$
 where the expression $I(m,u^\varepsilon_s,f_s,g_s)$ is defined in \cite{KR}, and by virtue of Lemma 2.1 of the same article it satisfies, 
 $$
 I(m,u^\varepsilon_s,f_s,g_s) \leq N(\|u^\varepsilon_s\|_m^2+\|f_s\|_m^2+\|g_s\|_{m+1}^2).
 $$
For the third term of the right hand side of the above inequality, by virtue of lemma \ref{lem: form composition},   we have  
$$
(\D^\alpha H_s(z)u^\varepsilon_s, \D^\alpha h_s(z) )_0=(\D^\alpha u^\varepsilon _s(x+c)- \D^\alpha u^\varepsilon_s(x), \D^\alpha h_s(z))_0
$$
$$
+\sum_{|\alpha|\leq m} \sum_{\rho=1}^d \sum_{\begin{subarray}{l}\beta \neq 0 \\ \beta \leq \alpha  \end{subarray}}  {\alpha \choose \beta } (\D^\beta c^\rho \D^{\alpha- \beta} \D_\rho u^\varepsilon (x+c)+u^{^\varepsilon (\alpha)}, \D^\alpha h_s(z))_0
$$
$$
\leq (\D^\alpha u^\varepsilon _s(x+c)- \D^\alpha u^\varepsilon_s(x), \D^\alpha h_s(z))_0
$$
$$
+ N|\bar{c}(z)|^2\|u^\varepsilon_s\|^2_m+ N\|h_s(z)\|^2_m.
$$
Then we notice that, 
 $$
 (\D^\alpha u^\varepsilon _s(x+c)- \D^\alpha u^\varepsilon_s(x), \D^\alpha h_s(z))_0
 $$
 $$
\sum_{i=1}^d  \int_0^1 \int_{\bR^d} \D_i \D^\alpha u^\varepsilon_s(T_{s,\theta ,z}(x))c^i_s(x,z) \D^\alpha h_s(z,x) dx d\theta
 $$
 $$
 -\sum_{i=1}^d  \int_0^1 \int_{\bR^d} \D_i \D^\alpha u^\varepsilon_s(x)\D_i (c^i_s(J_{s,\theta ,z}(x),z)
 $$
 $$
 \times \D^\alpha h_s(z,J_{s,\theta ,z}(x)) |det DJ_{s,\theta ,z}(x)|)dx d\theta
 $$
 $$
 \leq N|\bar{c}(z)|^2\|u^\varepsilon_s\|^2_m+N \|h_s(z)\|^2_{m+1}.
 $$
 By using this, and \eqref{eq: main estimate}, we obtain
 $$
E\sup_{t \leq r} \|u^\varepsilon_t\|^2_m\leq  \|\psi\|^2_m+N \int_0^r  E \|u^\varepsilon_t\|^2_m dt 
$$
$$
+NE\int_0^T \left(\|f_t\|^2_m+\|g_t\|^2_{m+1}+\int_Z \|h_t(z)\|^2_{m+1} \mu(dz)\right) dt
 $$
 $$
 +2E\sup_{t \leq r}\big|\sum_{k=1}^ \infty \int_0^t (u^\varepsilon_s, M^ku^\varepsilon_s+ g^k_s)_m \ dw^k_s\big|
 $$
 $$
 +E\sup_{t \leq r}\big|\sum_{|\alpha| \leq m} \int_0^t \int_Z\|\D^\alpha(u^\varepsilon_s(x+c))\|_0^2-\|\D^\alpha u^\varepsilon_s (x)\|_0^2  \tilde{N}(dz,ds) \Big|
 $$
 $$
 +E\sup_{t \leq r}\big|\sum_{|\alpha| \leq m} \int_0^t \int_Z  (\D^\alpha(u^\varepsilon_s(x+c)), \D^\alpha h_s(z))_0\tilde{N}(dz,ds) \Big|
 $$
 \begin{equation}                                \label{eq: after Ito's}
 +E\sup_{t \leq r}\big| \int_0^t \int_Z \| h_s(z)\|^2_m \tilde{N}(dz,ds)\Big| .
 \end{equation}
 Then as in \cite{KR}, we have for any $\delta>0$
$$
 2E\sup_{t \leq r}\big|\sum_{k=1}^ \infty \int_0^t (u^\varepsilon_s, M^ku^\varepsilon_s+ g^k_s)_m \ dw^k_s\big|
 $$
 $$
 \leq \delta E\sup_{t \leq r} \|u^\varepsilon_t\|^2_m+ 
 N\int_0^r (E\|u^\varepsilon_t\|^2_m+E\|g\|^2_{m+1} )dt,
 $$
 where $N=N(d,K,m,\delta)$. By virtue of Lemma \eqref{lem: form composition}  and Remark \ref{rem: |v|, |v|2},  we have, 
 $$
 \|\D^\alpha(u^\varepsilon_s(x+c))\|_0^2-\|\D^\alpha u^\varepsilon_s (x)\|_0^2 
 $$
 $$
 \leq \|\D^\alpha u^\varepsilon_s(x+c)\|_0^2-\|\D^\alpha u^\varepsilon_s (x)\|_0^2 +N|\bar{c}(z)|\|u^\varepsilon_s\|_m^2.
 $$
 Then we have
 $$
 \|\D^\alpha u^\varepsilon_s(x+c)\|_0^2-\|\D^\alpha u^\varepsilon_s (x)\|_0^2
 $$
 $$
 =\sum_{i=1}^d \int_0^1 \int_{\bR^d}2\D^\alpha u^\varepsilon_s(T_{t,\theta,z}(x))  \D_i \D^\alpha u^\varepsilon_s(T_{t,\theta,z}(x)) c^i_t(x,z) dx d\theta
 $$
 $$
 =-\sum_{i=1}^d\int_0^1 \int_{\bR^d} [\D^\alpha u^\varepsilon_s(x)]^2  \D_i \left( c^i_t(J_{t,\theta,z}(x),z)|det DJ_{t,\theta,z}(x)| \right) dx d\theta
 $$
 $$
 \leq N |\bar{c}(z)| \|u^\varepsilon_s\|^2_m.
 $$
 Hence, by the Burkholder-Gundy-Davis inequality, and Young's inequality, we get for any $\delta>0$, 
  $$
 +E\sup_{t \leq r}\big|\sum_{|\alpha| \leq m} \int_0^t \int_Z\|\D^\alpha(u^\varepsilon_s(x+c))\|_0^2-\|\D^\alpha u^\varepsilon_s (x)\|_0^2  \tilde{N}(dz,ds) \Big|
 $$
$$
\leq \sum_{|\alpha| \leq m}  E\left( \int_0^r \int_Z \left(\|\D^\alpha(u^\varepsilon_s(x+c))\|_0^2-\|\D^\alpha u^\varepsilon_s (x)\|_0^2\right)^2 \mu(dz) ds \right)^{1/2}
$$
 $$
 \leq \delta E\sup_{t \leq r} \|u^\varepsilon_t\|^2_m+ 
 N\int_0^r E\|u^\varepsilon_t\|^2_mdt,
 $$
 where $N=N(d,K,m,\delta)$.   Since 
 $$
 (\D^\alpha(u^\varepsilon_s(x+c)), \D^\alpha h_s(z))_0\leq N\|u^\varepsilon_s\|_m\|h_s(z)\|_m,
 $$
 by the Burkholder-Gundy-Davis inequality, and Young's inequality,  the term in the fifth line of \eqref{eq: after Ito's} can be estimated by 
 $$
 \delta E\sup_{t \leq r} \|u^\varepsilon_t\|^2_m+ 
 N\int_0^r \int_Z E\|h_t(z)\|^2_m\mu(dz)dt.
 $$
 Also the last term in \eqref{eq: after Ito's} can be estimated by 
 $$
 2E \int_0^T \int_Z \|h_t(z)\|^2_m \mu(dz)dt. 
 $$
 Combining these estimates we get
 $$
E\sup_{t \leq r} \|u^\varepsilon_t\|^2_m\leq N \|\psi\|^2_m+N \int_0^r  E \|u^\varepsilon_t\|^2_m dt 
$$
$$
+NE\int_0^T \left(\|f_t\|^2_m+\|g_t\|^2_{m+1}+\int_Z \|h_t(z)\|^2_{m+1} \mu(dz)\right) dt< \infty ,
 $$
 From which, by virtue of Gronwall's lemma we obtain \eqref{eq: estimate for the solution} for $u^\varepsilon$, assuming that $a^{ij}$ are smooth. For the general case, we  mollify $a^{ij}$ to obtain $a^{ij(n)}$, and let us call $u^{\varepsilon (n)}$ the solution of \eqref{eq:main equation}, with $a^{ij}$ replaced by $a^{ij(n)}+ \varepsilon \delta_{ij}$, where $\delta_{ij}$ is the Kronecker delta. Then we have that estimate \eqref{eq: estimate for the solution} holds for $u^{\varepsilon (n)}$. Also, the difference $u^{\varepsilon (n)}-u^\varepsilon$ satisfies 
 \eqref{eq:main equation}, with $a^{ij}$ replaced by $a^{ij(n)}+ \varepsilon \delta_{ij}$, Assumptions \ref{as: non-de regularity}- \ref{as: non-de parabolicity} in force, the constants appearing there independent of $n\in \mathbb{N}$, and $f=( a^{ij(n)}- a^{ij}) \D_{ij}u^\varepsilon$, $g=0$, $h=$ and $\psi=0$. Hence, by \eqref{eq: main estimate non-de}, we obtain $E\sup_{t \leq T} \|u^\varepsilon_t-u^{\varepsilon (n)}_t\|^2_m \to 0$, as $n \to \infty$, which shows that estimate \eqref{eq: estimate for the solution} holds for $u^\varepsilon$ for the general case. Then, once estimate \eqref{eq: estimate for the solution} is obtained for $u^\varepsilon$, one can conclude the proof in the same way as in \cite{KR}, \cite{KR1}. One can find a sequence $\varepsilon_n\to 0$, and a function $u \in \mathfrak{H}^m$ such that $u^{\varepsilon_n}$ converges weakly to $u$, and $u$ is a solution of \eqref{eq:main equation}. It follows that $u$ is c\'adl\'ag in $H^{m-1}$.  Then we can take a sequence $u^n$ of convex combinations of $u^{\varepsilon_n}$ such that $\|u^n- u\|_m \to 0$ for almost every $(\omega,t) \in \Omega\times [0,T]$. Hence we can find $\mathbb{T}$,  a dense countable set of $[0,T]$ such that, almost surely, $\|u^n_t- u_t\|_m \to 0$ as $n \to \infty$, for all $t \in \mathbb{T}$. Let $\mathbb{L}$ be a countable dense subset of $H^m$ consisting of smooth functions with compact support. Since for any multi-index $\gamma$ of order $m$, and any $\phi \in \mathbb{L}$, the expression $( u_t, \D^\gamma\phi)_0$ is c\'adl\'ag in $t$, 
 we have that almost surely 
 $$
 \sup_{\phi \in \mathbb{L}} \sup_{t < T} \frac{( u_t, \D^\gamma \phi)_0}{\|\phi\|_0 } \leq  \sup_{\phi \in \mathbb{L}} \sup_{t \in \mathbb{T}} \frac{(u_t, \D^\gamma  \phi)_0 }{\|\phi\|_0 }
 $$
$$
   \leq \liminf_{n\to \infty}  \sup_{t \leq T} \|\D^\gamma u^n_t\|_0.
  $$
  The right hand side of of the above  inequalities is finite, which implies that almost surely $\D^\gamma u_t \in H^0$ for any $t< T$, and the following holds
  \begin{equation}                      \label{eq: norm of u in m}
  \sup_{t<T}\| \D^\gamma u_t\|_0 \leq \liminf_{n\to \infty}  \sup_{t \leq T} \|\D^\gamma u^n_t\|_0.
   \end{equation}
 For  $t=T$ we proceed similarly. By virtue of the main estimate, we can take a subsequence $\varepsilon_{n(k)}$ of $\varepsilon_n$, and a function $\bar{u}_T \in L_2(\Omega, \mathscr{F}_T; H^m)$ such that $u^{\varepsilon_{n(k)}}_T$ converges weakly to $\bar{u}_T $. Then for any $\phi \in C^\infty_c$, we have 
 $$
 (u^{\varepsilon_{n(k)}}_T, \phi)_0=R(T,a^{ij}+\varepsilon_{n(k)}\delta_{ij}, u^{\varepsilon_{n(k)}}, \phi).
 $$
 Since the integral and the stochastic integral are continuous linear operators from $\mathfrak{H}^0$ into $L_2(\Omega,\mathscr{F}_T)$, therefore weakly continuous, by letting $k \to \infty$ we obtain
 $$
 (\bar{u}_T, \phi)_0=R(T,a^{ij}, u, \phi)=(u_T,\phi)_0.
 $$
 Hence $u_T=\bar{u}_T$ almost surely. It also follows that 
 $$
 \|u_T\|_m \leq \liminf_{n \to \infty} \|u^n_T\|_m. 
 $$
 This combined with \eqref{eq: norm of u in m} and Fatou's lemma leads to estimate \eqref{eq: estimate for the solution} for $u$. To show that $u$ is weakly c\'adl\'ag we proceed as follows. We  have that for any $\phi\in \mathbb{L}$, the expression $(u_t,\phi)_m$ is c\'adl\'ag.  We also have that $\sup_{t\leq T}\|u_t\|_m< \infty$. It follows then that for any $v \in H^m$, $(u_t,v)_m$ is right continuous. To show the existence of left limits, we fix $t \in [0,T]$ and $v \in H^m$,  and let $v_k\in \mathbb{L}$, such that $\|v_k-v\|_m \to 0$. Then for any multi-index $\gamma$ with $|\gamma| = m$, we have
 $$
 \sup_k |(\D^{\gamma-e_i}u_{t-}, \D^{\gamma+e_i} v_k)_0| = \sup_k \lim_{t_n \uparrow t}|(\D^{\gamma-e_i}u_{t_n}, \D^{\gamma+e_i} v_k)_0|
 $$
 $$
  \leq \sup_{t \leq T} \|u_t\|_m \|v\|_m< \infty,
 $$
 for an appropriate $i \in \{1,...,d\}$.
 Hence there exists a subsequence $k(l)$ and $q \in \bR$ such that 
 $$
 -(\D^{\gamma-e_i}u_{t-}, \D^{\gamma+e_i} v_{k(l)})_0 \to q, \ \text{as} \ l \to \infty.
 $$
 We claim that for  any $\varepsilon>0$, there exists $\delta>0$, such that if $0<t-s< \delta$ then 
 $|q-(\D^{\gamma}u_s, \D^{\gamma} v)_0| \leq \varepsilon$.
 We have
 $$
 |q-(\D^{\gamma}u_s, \D^{\gamma} v)_0|  \leq |(\D^{\gamma}u_s, \D^{\gamma} v_{k(l)})_0-(\D^{\gamma}u_s, \D^{\gamma} v)_0|
 $$
$$
+|(\D^{\gamma-e_i}u_s, \D^{\gamma+e_i} v_{k(l)})_0-(\D^{\gamma-e_i}u_{t-}, \D^{\gamma+e_i} v_{k(l)})_0|
$$
$$
+|(\D^{\gamma-e_i}u_{t-}, \D^{\gamma+e_i} v_{k(l)})_0+q|
$$
$$
\leq \sup_{t \leq T} \|u_s\|_m\|v_{k(l)}-v\|_m
$$
$$
+|(\D^{\gamma-e_i}u_s, \D^{\gamma+e_i} v_{k(l)})_0-(\D^{\gamma-e_i}u_{t-}, \D^{\gamma+e_i} v_{k(l)})_0|
$$
$$
+|(\D^{\gamma-e_i}u_{t-}, \D^{\gamma+e_i} v_{k(l)})_0+q|
$$
Hence one can take $l$ large enough and then choose a sufficiently small $\delta$. This finishes the proof.
 \end{proof}
 \section*{Acknowledgement}
 The author would like to express  gratitude towards his Phd advisor, Professor Istv\'an Gy\"ongy, for his guidance and his useful suggestions.

\end{document}